\documentclass[reqno]{amsart}
\usepackage{amssymb}
\usepackage{amsmath}

\makeatletter
\@addtoreset{equation}{section}
\makeatother

\renewcommand\thefigure{\thesection.\@arabic\c@figure}
\renewcommand\thetable{\thesection.\@arabic\c@table}

\newtheorem{theorem}{Theorem}[section]
\newtheorem{lemma}[theorem]{Lemma}
\newtheorem{proposition}[theorem]{Proposition}
\newtheorem{corollary}[theorem]{Corollary}

\newcommand{\mc}[1]{{\mathcal #1}}
\newcommand{\mf}[1]{{\mathfrak #1}}

\newcommand{\bb}[1]{{\mathbb #1}}

\def\Z{\mathbb Z}

\def\F{\mathcal F}

\begin{document}

\author{Magda Peligrad}

\address{\noindent Department of Mathematical Sciences\\
University of Cincinnati\\
PO Box 210025\\
Cincinnati, OH 45215
\newline
e-mail:  \rm \texttt{peligrm@math.uc.edu} }

\author{Sunder Sethuraman}

\address{\noindent 396 Carver Hall\\
Department of Mathematics\\
Iowa State University\\
Ames, IA \ 50011
\newline
e-mail:  \rm \texttt{sethuram@iastate.edu} }

\title[fBM limits in SSEP]
{On fractional Brownian motion limits in one dimensional
nearest-neighbor symmetric
simple exclusion}

\begin{abstract}
A well-known result with respect to the one dimensional
nearest-neighbor symmetric simple exclusion process is the convergence to
fractional Brownian motion with Hurst parameter $1/4$, in the sense
of finite-dimensional distributions, of the subdiffusively
rescaled current across
the origin, and the subdiffusively rescaled tagged particle position.

The purpose of this note
is to improve this convergence to a functional central limit
theorem, with respect to the uniform topology, and so complete the
solution
to a
conjecture in the literature with respect to simple exclusion processes.
\end{abstract}



\thanks{Research supported in part by NSA-H982300510041,
NSF-DMS-0504193, NSA-H982300710016, and a University of Cincinnati Taft grant}

\maketitle

\section{Introduction}

Informally, the one dimensional
nearest-neighbor symmetric simple exclusion process
follows a collection of random walks on
the lattice $\Z$ which move independently except in that
jumps to already occupied sites are suppressed.
More precisely, the exclusion model is a Markov process $\eta_t=
\{\eta_t(x): x\in \Z\}$ evolving on the configuration space $\Sigma = \{0,1\}^\Z$ with generator,
$$(L\phi)(\eta) \ = \ \frac{1}{2}\sum_{x} \big [ \phi(\eta^{x,x+1}) -
 \phi(\eta)\big]$$
where $\eta^{x,x+1}$ is the configuration obtained from $\eta$ by
 exchanging the values at $x$ and $x+1$,
$$\eta^{x,x+1}(z) \ = \ \left\{\begin{array}{rl}
\eta(z) & \ {\rm when \ }z\neq x,x+1\\
\eta(x) & \ {\rm when \ }z=x+1\\
\eta(x+1)&\ {\rm when \ }z=x.\end{array}\right.$$ A more formal
treatment can be found in Liggett \cite{Liggett}. Later, in Section
2, we will also give Harris's description of the model in terms of a
``stirring process.''

As the process is ``mass conservative,'' that is no birth or death,
one expects a family of invariant measures corresponding to particle
density.  In fact, for each $\rho\in [0,1]$, the product over $\Z$
of Bernoulli measures $\nu_\rho$ which independently puts a particle
at locations $x\in \Z$ with probability $\rho$, that is
$\nu_\rho(\eta_x=1)=1-\nu_\rho(\eta_x=0)=\rho$, are invariant (cf.
Liggett \cite{Liggett}).

\medskip
In this note, we concentrate on the integrated flux across the
origin $J(t)$, and the position of a tagged, or distinguished
particle $X(t)$, say the first particle to the left of $1/2$, when
initially the exclusion process starts in an equilibrium $\nu_\rho$
for $0<\rho<1$.  Both objects are interestingly connected, and have
been long well-studied in the literature (see Section 8.4 in Liggett
\cite{Liggett}, Section 6.4 in Spohn \cite{Spohn}, De Masi-Ferrari
\cite{DeMF}).

Perhaps the most intriguing behavior of the current
and tagged particle is their subdiffusive fluctuation behavior,
explained physically in part by the enforced ordering of particles
with no leapfrogging allowed in the dynamics.  It was shown in
Arratia \cite{Arratia}, Rost-Vares \cite{RV} and De Masi-Ferrari
\cite{DeMF} that
\begin{equation}
 \label{subdifflimits}t^{-1/4} J(t) \ \stackrel{d}{\rightarrow} \
N(0,\sigma_J^2) \ \ {\rm and \ \ } t^{-1/4}X(t) \
\stackrel{d}{\rightarrow} \ N(0,\sigma^2_X)\end{equation} where
$\sigma^2_J = \sqrt{2/\pi}(1-\rho)\rho$ and $\sigma^2_X=
\sqrt{2/\pi}(1-\rho)\rho^{-1}$. This is in contrast to the diffusive
behavior in higher dimensions or when the jump probability is longer
range (cf. Chapter 6 in Spohn \cite{Spohn}, Part III in Liggett
\cite{Liggett2}, Sethuraman \cite{seth}).

Given these results, the general belief (see Conjecture 6.5 in Spohn
\cite{Spohn}) is that the process limits with respect to the
rescaled current and tagged particle position converge to respective
fractional Brownian motions with Hurst parameter $1/4$.  In fact,
straightforward modifications of the arguments for
(\ref{subdifflimits}) give convergence in the sense of
finite-dimensional distributions (a case of Theorem 1.2 Landim-Volchan
\cite{LV} gives a
specific statement; see also Jara-Landim \cite{JL}). However, it appears the full functional central
limit theorem conjectured in Spohn \cite{Spohn}, with respect to
exclusion processes, has not been addressed.

The aim of this article is to complete the proof of this
conjecture by supplying path tightness estimates to show, as
$\lambda \uparrow \infty$, that both
\begin{equation}
\label{mainlimits}
\sigma^{-1}_J\lambda^{-1/4}J(\lambda t) \ \Rightarrow \ \bb B_{1/4}(t), {\ \rm and \ \ }
  \sigma^{-1}_X\lambda^{-1/4}X(\lambda
t) \ \Rightarrow \ \bb B_{1/4}(t)
\end{equation} where $\bb B_{1/4}(t)$ is the
standard fractional Brownian motion with parameter $1/4$, and
$\Rightarrow$ denotes weak convergence in $D([0,1])$ endowed with the
  uniform topology.

The plan of the paper is to give some preliminary representations and
estimates in Section \ref{prelim}, and then deduce the limits
(\ref{mainlimits}), through certain maximal inequalities and discrete
time process approximations, in Corollaries \ref{J_fbm} and \ref{tagged_fbm} in
Sections \ref{J_sect} and \ref{tagged_sect} respectively.

Throughout, unless otherwise clear, $P=P_{\nu_\rho}$ and $E=E_{\nu_\rho}$ denote the process
measure and expectation starting under the equilibrium $\nu_\rho$.
Also, as standard, $\lfloor x\rfloor$ denotes the integer part of
$x$.

\section {Representations of current and tagged particle}
\label{prelim}
 In this section, we state a convenient construction of the exclusion
 process through a ``stirring process'' first introduced by Harris \cite{Harris}, and
 discuss some representations of the current and tagged particle position.

 \subsection{ Stirring process}
 \label{stirring_subsection}  The stirring process $\xi_{t}^{i}\in \Z$ for $i\in \Z$ is defined as
 follows.
 At time $t=0$, a particle is put at
 each site and we define $\xi_{0}^{i}=i$ for each $i\in \Z$. To each bond
 $(x,x+1)$ with $x\in \Z$, we associate a Poisson process (clock) with parameter $1/2$.
 When the clock rings at bond $(x,x+1)$, the particles at these sites
 interchange their positions. Then, $\xi_{t}^{i}$ is the position at time $t$ of
 the particle which was at $i$ at time $0$. Given an initial
 configuration
 $\eta$, the simple exclusion process, in terms of
 the stirring process, is
 \begin{equation*}
 \eta _{t}(x)\ =\ 1\big\{x\in \{\xi_{t}^{i}: \eta (i)=1\}\big\},
 \end{equation*}
 that is, in words, $\eta_t(x) =1$ if and only if there is an $i\in \Z$ so that
 $\xi_t^i = x$ and $\eta(i)=1$.
 \subsection{ Current representations}
 \label{current_subsection}  Let $N_+(t), N_-(t)$ be counting
 processes, with infinitessimal rates $(1/2)\eta_s(0)(1-\eta_s(1))$, $(1/2)\eta_s(1)(1-\eta_s(0))$, which count the number of
 particles which cross $0\rightarrow 1$ and $1\rightarrow 0$ up to time
 $t$ respectively.  Then, the current across the bond $(0,1)$ up to
 time $t$ is given by
 $$J(t) \ = \ N_+(t) - N_-(t).$$
 As $N_+(t) - \frac{1}{2}\int_0^t \eta_s(0)(1-\eta_s(1))ds$ and
 $N_-(t) - \frac{1}{2}\int_0^t
 \eta_s(1)(1-\eta_s(0))ds$ are martingales,
 we note the useful decomposition
 $$J(t) \ = \ M(t) + A(t)$$
 where $M(t) = N_+(t) - N_-(t) - \frac{1}{2}\int_0^t
 (\eta_s(0) - \eta_s(1)) ds$ is a martingale, and $A(t) =
 \frac{1}{2}\int_0^t (\eta_s(0)-\eta_s(1))ds$ (cf. (III.2.37) in
 Liggett \cite{Liggett2}).

 In terms of the stirring process, following
 the development in De Masi-Ferrari \cite{DeMF},
 define
 \begin{equation*}
 K^{+}(t)\ =\ \sum_{i\leq 0}1\{\xi_{t}^{i}>0\}; \ \
 K^{-}(t)\ =\ \sum_{i>0}1\{\xi_{t}^{i}\leq 0\}.
 \end{equation*}
 In words, $K^{+}(t), K^-(t)$ are the number of stirring particles
 starting to the
 left and right of the point $1/2$ and sitting at the right
 and left of $1/2$ at time $t$ respectively.  Denote
$$U_i(t) \ =\ \left\{\begin{array}{rl}
 1\{\xi^i_t>0\} &\ {\rm when} \ i\leq 0\\
1\{\xi^i_t\leq 0\}&\ {\rm when \ } i>0.\end{array}\right.$$

 As in the stirring process all sites are always occupied,
 each crossing of the bond $(0,1)$ in one direction corresponds to a
 simultaneous crossing in the opposite direction.
 Then, $K^{+}(t)- K^-(t)$ is constant in $t$, and since $K^{+}(0)=K^-(0)=0$, $K^{+}(t)=K^-(t):=K(t)$, for all $t\geq 0$. Hence,
 \begin{equation*}
 J(t)=\sum_{i\leq 0}1\{\xi_{t}^{i}>0\}\eta (i)-\sum_{i>0}1\{\xi_{t}^{i}\leq
 0\}\eta (i).
 \end{equation*}

 For $K(t)\geq 1$, let $i_{1}<i_{2}<...<i_{K(t)}\leq 0$ be the random locations for which $\xi
 _{t}^{i_{k}}>0$, and $0<j_{1}<j_{2}<...<j_{K(t)}$ be the random locations for
 which $\xi _{t}^{j_{k}}\leq 0$. Define $B_{k}^{+}=\eta(i_{k})$ and
 $B_{k}^{-}=\eta(j_{k})$ and $A_{k}(t)=B_{k}^{+}-B_{k}^{-}$. Then,
 with the convention that the sum from $1$ to $0$ is equal to $0$, we have the representation
 \begin{equation*}
 J(t)=\sum_{k=1}^{K(t)}A_{k}(t).
 \end{equation*}

We now state some known facts and consequences.

\medskip  (a) Clearly, given $K(t)$, and the
random locations $\{i_{k}\}$ and $\{j_{k}\}$ with $1\leq k\leq K(t)$
the variables $\{A_{k}(t)\}$ are independent, identically
distributed,
 mean $0$, and
 $$P(A_{k}(t) = 1) =
 P(A_k(t) = -1) = \rho(1-\rho), \ \ {\rm and \ \ }
 P(A_k(t) = 0 ) = 1-2\rho(1-\rho).$$

 (b) Also $K(t)$ is a sum of negatively correlated $0,1$ valued random
 variables. Moreover, by Lemma 4.12 Liggett \cite{Liggett},
for all finite $T\subset
 \Z$,
 and all $A\subset \Z$,
\begin{equation}
 P\Big(\bigcap_{i\in T}(\xi _{t}^{i}\in A)\Big)\ \leq \ \prod_{i\in T}P\Big(\xi _{t}^{i}\in A\Big).
 \label{neg}
\end{equation}

 (c) From basic considerations, $E[K(t)]=E(z(0,t)_{+}),$ where $z(0,t)$ is a symmetric
 random walk on $\Z$ starting at the origin.  Then, $E[K(t)] \leq \sqrt{t}$, and
 \begin{equation*}
 \lim_{t\rightarrow \infty }\frac{E[K(t)]}{\sqrt{t}}\ =\ \frac{1}{\sqrt{2\pi }}.
 \end{equation*}

 (d) We have also some $p$-moment estimates.  Denote $\|V\|_p =
 (E[V^p])^{1/p}$ for simplicity.   \begin{lemma}
 \label{K_p_moment}
 For integers $p\geq 1$, there is a constant $C_0=C_0(p)<\infty$ so
 that for all $t\geq 1$,
 \begin{equation*}
 \|K(t)\|_p\ \leq \ C_0\sqrt{t}.
 \end{equation*}
 \end{lemma}

\begin{proof}
 First, as $K(t) = \sum_{i\leq 0}U_i(t)$, and by (\ref{neg}), we have that
 $$\|K(t)\|_p\ \leq \ \Big\|\sum_{i\leq 0} \tilde{U}_i\Big\|_p$$
 where $\tilde{U}_i\stackrel{d}{=} U_i(t)$ and $\{\tilde{U}_i\}$ are
 $0,1$ valued independent random variables.

 To further estimate, we use Rosenthal inequality (Theorem 1.5.9 in De
 la Pena-Gin\'e \cite{PG}) on independent and nonnegative random variables
 $\{\beta_i\}$:  For $p\geq 1$, there exists a constant $C_1=C_1(p)<\infty$ such that
 $$\Big\|\sum \beta_i\Big\|_p \ \leq \ C_1 \max\Big\{\sum E[\beta_i], \Big(\sum
 E[\beta_i^p]\Big)^{1/p}\Big\}.$$
Now, applying Rosenthal's inequality, for a positive integer
 $p$,  noting that $\tilde{U}_i^p = \tilde{U}_i$, we have,
 \begin{equation*}
\Big \|\sum_{i\leq 0}\tilde{U}_{i}\Big\|_{p}\ \leq \ C_1\Big(\sum_{i\leq 0}E[
 \tilde{U}_{i}]+\Big[\sum_{i\leq
   0}E[\tilde{U}_{i}^{p}]\Big]^{1/p}\Big)\ \leq
 \ C_1\Big(E[K(t)]+(E[K(t)])^{1/p}\Big).
\end{equation*}
The result follows from properties of $E[K(t)]$ listed in part (c), and that $p\geq 1$.
  \end{proof}

 \subsection{ Tagged particle representations}
 \label{tagged_subsection}  Consider a distinguished particle in the
 exclusion system.
 One representation for its displacement $Z(t) = X(t)-X(0)$ is through
 the ``Lagrangian frame,'' $\zeta_t = \theta_{Z(t)}\eta_t$ where
 $\theta_y \eta$ is the shifted configuration $(\theta_y \eta)
 (x) = \eta_{x+y}$.  Then, $\zeta_t$ is a Markov process on $\Sigma' =
 \{\zeta\in \Sigma: \zeta(0)=1\}$ with generator
 \begin{eqnarray*}
 (\mc L\phi)(\zeta) &=&  \frac{1}{2}\sum_{x\neq
   -1,0}\big(\phi(\zeta^{x,x+1})-\phi(\zeta)\big)\\
 &&\ + \ \frac{1}{2}\sum_{i=-1,1}(1-\zeta_i)\big(\phi(\tau_i\zeta)-\phi(\zeta)\big)
 \end{eqnarray*}
 where $\tau_k\zeta$ is the configuration obtained by displacing the
 tagged particle $k$ steps and then shifting the frame,
 $$(\tau_k\zeta)(x) \ =\  \left\{\begin{array}{rl}\zeta(x+k)&\ {\rm
       when}\ x \neq 0, -k\\
 \zeta(0)& \ {\rm when}  \ x=0\\
 \zeta(k)&\ {\rm when \ } x = -k.
 \end{array}\right.
$$
 Define $\mc N_+(t), \mc N_-(t)$ as the counting processes, with
 infinitessimal rates $(1/2)(1-\zeta_s(1))$, $(1/2)(1-\zeta_s(-1))$, which count
 the number of frame shifts of size $1$ and $-1$ respectively up to
 time $t$. Then,
 $$X(t) -X(0)\ = \ \mc N_+(t) - \mc N_-(t).$$
 Similar to the current representation, $\mc N_+(t)-
 \frac{1}{2}\int_0^t (1-\zeta_s(1))ds$ and $\mc N_-(t)
 -\frac{1}{2}\int_0^t (1-\zeta_s(-1))ds$ are martingales, and
 $$X(t) -X(0)\ = \ \mc M(t) + \mc A(t)$$
 where
 $\mc M(t) = \mc N_+(t) - \mc N_-(t) - \frac{1}{2}\int_0^t
 (\zeta_s(-1)-\zeta_s(1))ds$ is a martingale and $\mc A(t) =
 \frac{1}{2}\int_0^t (\zeta_s(-1)-\zeta_s(1))ds$ (cf. Proposition
 III.4.1 in Liggett \cite{Liggett2}).

 On the other hand, with
 respect to the stirring process and a configuration
 $\eta$ drawn from $\nu_\rho$,
 following
 D\"urr-Goldstein-Lebowitz \cite{DGL} and the exposition in De
 Masi-Ferrari \cite{DeMF},
 for $k\geq 1$ let $Y_k(t)$ be the position of the $k$th particle of
 $\eta_t$ to the right of $1/2$; for $k\leq 0$ let
 $Y_k(t)$ be the position of the $(|k|+1)$th particle of $\eta_t$ to the
 left of $1/2$.  Then, at time $t$, the tagged particle, initially the
 $0$th labeled particle, is the $J(t)$th
 particle,
 $$X(t) \ = \ Y_{J(t)}(t),$$ where $J(t)$ was defined in Subsection 2.2.

It will also be useful to note, under the invariant measure
$\nu_\rho$, that $Y_n(t)
\stackrel{d}{=} Y_n(0)$,
$$
Y_n(t) \ =\  \left\{\begin{array}{rl}
Y_1(t) + \sum_{i=1}^{n-1} d_i(t) &\ {\rm for \ } n\geq 2\\
Y_0(t) - \sum_{i=n}^{-1} d_i(t) & \ {\rm for \ } n\leq -1\end{array}
\right.$$ where $d_i(t) = Y_{i+1}(t)-Y_{i}$ is the spacing between
the $i$th and $(i+1)$th particles, and also,
$\{d_i(t): i\neq 0\}$, $Y_1(t)$ and $|Y_0(t)|+1$ have independent
Geometric$(\rho)$ distributions.



\section{Tightness and fBM limit for the current}
\label{J_sect}
In this section, we prove the following theorem which is the main
vehicle in the article.  At the end of the section, we state as
Corollary \ref{J_fbm} the fractional Brownian motion invariance
principle for the current.

\begin{theorem}
\label{J_tight} Under initial distribution $\nu_\rho$, the
stochastic process $\lambda^{-1/4}J(\lambda t)$ is tight in
$D([0,1])$ endowed with the uniform topology.
\end{theorem}

\begin{proof}  The proof follows from showing
that the discretized version $\lambda^{-1/4}J(\lfloor\lambda t\rfloor)$ is tight
in $D([0,1])$ in the uniform norm (Proposition \ref{J_disc}), and
then that the difference with the desired process is negligible
(Proposition \ref{J_diff}).
\end{proof}


The first step is to state a useful maximal inequality.
\begin{lemma}
\label{max_ineq}
For integers $m\geq 1$, and even integers $p\geq 2$,
\begin{equation*}
E[J^{p}(m)]\ \leq \ C_2m^{p/4}
\end{equation*}
where $C_2=C_2(p)$ is a constant.  Moreover, for even integers
 $p\geq 6$, there is a constant $C_3=C_3(p)$ such that
\begin{equation*}
E\Big[\max_{1\leq i\leq m}J^{p}(i)\Big]\ \leq \ C_3m^{p/4}.
\end{equation*}
\end{lemma}

\begin{proof}
We recall Marcinkiewicz inequality (Lemma 1.4.13 in De la
Pena-Gin\'e \cite{PG}): For $p\geq 1$, there exists a constant
$C_4=C_4(p)$ such that for centered, independent $L^p$ random
variables $\{\beta_i\}$,
$$E\Big|\sum \beta_i\Big|^p \ \leq \ C_4 E\Big[\sum \beta^2_i\Big]^{p/2}.$$

Denote by $\F$ the $\sigma$-algebra generated by $K(t)$, and the
random locations  $\{i_{k}\}$ and $\{j_{k}\}$ with $1\leq k\leq
K(t)$ introduced in Subsection 2.2. Then, by conditioning first on
$\F$, taking into account that $A_k(m)^2\leq 1$ for all $k,m$, and
properties in part (a) Subsection 2.2, we have
\begin{eqnarray*}
E[J^{p}(m)]&=&E\Big[
E\Big[\Big(\sum_{1\leq k\leq K(m)}A_{k}(m)\Big)^p\Big|\F\Big]\Big] \\
&\leq &C_4E\Big[\Big(\sum_{1\leq k\leq K(m)}A_{k}(m)^{2}\Big)^{p/2}\Big]\\
&\leq &C_4E[K^{p/2}(m)].
\end{eqnarray*}
The first statement now follows by Lemma \ref{K_p_moment}.

For the maximal inequality, we note first, from stationarity of
$J(t)$ and the proven first inequality, for all integers $0\leq
i\leq j\leq m$,
\begin{equation}
\label{J_stat} E\big[(J(j)-J(i))^{p}\big] \ = \
E\big[J^p(j-i)\big] \ \leq \ C_2(j-i)^{p/4}.
\end{equation}

We now recall a case of Theorem 3.1 in Moricz-Serfling-Stout \cite{MSS}:  Let
$S_{i,j} = \sum_{k=i}^{j} \beta_k$ where $\{\beta_k\}$ are arbitrary
random variables.  Let also $\mu\geq 1$ and $\alpha>1$.
Suppose
 for some nonnegative numbers $\{u_k\}$,
$E|S_{i,j}|^\mu \leq (\sum_{k=i}^{j} u_k)^\alpha$ for all $1\leq
i\leq j\leq n$.  Then, there is a constant $C_5=C_5(\mu,\alpha)$ such that
\begin{equation}
\label{serfling}
E\big[\max\{|S_{1,1}|,\ldots,|S_{1,n}|\}^\mu\big] \ \leq \
C_5\bigg(\sum_{k=1}^{n} u_k\bigg)^\alpha.\end{equation}

Then, as $p/4>1$, applying (\ref{serfling}) with respect to
(\ref{J_stat}), the second statement follows.
\end{proof}

We now consider the discretized process.
\begin{proposition}
\label{J_disc} Under initial distribution $\nu_\rho$, the stochastic
process $\lambda^{-1/4}J(\lfloor\lambda t\rfloor)$ is tight in $D([0,1])$
endowed with the uniform topology.
\end{proposition}

\begin{proof} According to Billingsley \cite{B},
a well-known tightness condition is to show, for all $\varepsilon >0$, that
\begin{equation*}
\lim_{\delta \rightarrow 0}\limsup_{\lambda\rightarrow \infty
}P\Bigg(\sup_{\stackrel{|s-t|<\delta}{s,t\in [0,1]}
}\lambda^{-1/4}|J(\lfloor\lambda t\rfloor)-J(\lfloor\lambda s\rfloor)|\geq \varepsilon \Bigg
)\ =\ 0
\end{equation*}
which reduces, in our situation of stationary increments, to proving
\begin{equation*}
\lim_{\delta \rightarrow 0}\limsup_{\lambda \rightarrow \infty
}\delta^{-1}P\Bigg(\sup_{s\in
[0,\delta]}\lambda^{-1/4}|J(\lfloor\lambda s\rfloor)|\geq \varepsilon \Bigg )\
=\ 0.
\end{equation*}

By Chebychev's inequality,
\begin{equation*}
P\Bigg(\sup_{s\in [0,\delta] }\lambda^{-1/4}|J(\lfloor\lambda
s\rfloor)|\geq \varepsilon \Bigg)\ \leq \ {\varepsilon
}^{-6}E\Bigg(\sup_{s\in [0,\delta]
}\lambda^{-6/4}|J^6(\lfloor\lambda s\rfloor)|\Bigg).
\end{equation*}
Also, by the maximal inequality in Lemma \ref{max_ineq},
\begin{equation*}
E\Bigg(\sup_{s\in [0,\delta]}\lambda^{-6/4}|J^6(\lfloor\lambda
s\rfloor)|\Bigg)\ = \ E\Bigg(\max_{1\leq i\leq \lfloor\lambda
\delta\rfloor}\lambda^{-6/4}|J^6(i)|\Bigg) \ \leq \ C_3\delta^{3/2}
\end{equation*}
which is enough to conclude the proof. \end{proof}

The difference between the discretized and desired process is handled
as follows.
\begin{proposition}
\label{J_diff}
\begin{equation*}
\lim_{\lambda\rightarrow \infty}P\Big(\sup_{0\leq t\leq
1}\lambda^{-1/4}|J(\lambda t)-J(\lfloor\lambda t\rfloor)|>\varepsilon \Big)\ = \
0.
\end{equation*}
\end{proposition}

\begin{proof}
By stationary increments, noting $(nt)-\lfloor nt\rfloor\leq 1$,
\begin{eqnarray*}
&&P\Big(\sup_{0\leq t\leq 1}\lambda^{-1/4}
|J(\lambda t)-J(\lfloor\lambda t\rfloor)|>\varepsilon \Big)\\
&&\ \ \ \ \ \ \ \ \ \ \ \ \ \ \ \ \leq \
P\Bigg(\sup_{\stackrel{0\leq s\leq t\leq \lambda}{t-s\leq
  1}}\lambda^{-1/4}
|J(t)-J(s)|>\varepsilon \Bigg) \\
&&\ \ \ \  \ \ \ \ \ \ \ \ \ \ \ \ \leq \
3\sum_{i=0}^{\lfloor\lambda\rfloor}P\Big(\sup_{i\leq t\leq
  i+1}\lambda^{-1/4}
|J(t)-J(i)|>\varepsilon/3 \Big)\\
&&\ \ \ \  \ \ \ \ \ \ \ \ \ \ \ \ =\
3(\lfloor\lambda\rfloor+1)P\Big(\sup_{0\leq t\leq
1}|J(t)|>\varepsilon \lambda^{1/4}/3\Big)
\\
&&\ \ \ \  \ \ \ \ \ \ \ \ \ \ \ \ \leq \
6E\Bigg[\sup_{0\leq t\leq 1}|J(t)|^{4}I\Big(\sup_{0\leq
t\leq 1}|J(t)|>\varepsilon \lambda^{1/4}/3\Big)\Bigg].
\end{eqnarray*}

We now show that  $E[\sup_{0\leq t\leq
1}|J(t)|^{4}]<\infty$ to deduce that the last quantity vanishes as
$\lambda\uparrow \infty$.
 Indeed, from the decomposition $J(t) = M(t) + A(t)$ where $M(t)$ is a
 martingale and $A(t) = (1/2)\int_0^t (\eta_0(s) - \eta_1(s)) ds$,
as the integrand $|\eta_0-\eta_1|$ is bounded by $1$, we need only
bound
$E[\sup_{0\leq t\leq 1}M^4(t)]$.  By using
Doob's inequality and simple computations,
 $$E\Big[\sup_{0\leq t\leq 1} M^4(t)\Big]
\ \leq \ (4/3)^4E[M^4(1)] \ \leq \
 4(4/3)^4\Big(E[J^4(1)] + 1\Big ).$$
 But, $J(t)= N_+(t)-N_-(t)$ is the difference of two counting
 processes each
with rates bounded by $1/2$, and so by coupling with respect to dominating Poisson rate$(1/2)$ processes,
we have $E[J^4(1)]<\infty$; this can
 also be seen from computing directly with $J(1) =
 \sum_{k=1}^{K(1)}A_k(1)$.
\end{proof}

\begin{corollary}
\label{J_fbm} Under initial distribution $\nu_\rho$, with respect to
the uniform topology on $D([0,1])$, we have as $\lambda\uparrow
\infty$,
\begin{equation*}
\lambda^{-1/4}J(\lambda t)\ \Rightarrow \ \sigma_J \bb B_{1/4}(t)
\end{equation*}
where $\bb B_{1/4}(t)$ is the fractional Brownian motion process
with index $1/4$.
\end{corollary}

\begin{proof}  From Theorem \ref{J_tight}, we know that
$\lambda^{-1/4}J(\lambda t)$ is tight.  Also, from the literature
(cf. Landim-Volchan \cite{LV}) the finite-dimensional distributions
of any limit are Gaussian. Hence, if $W(t)$ is a limit along a
subsequence, this limit is a continuous Gaussian process. By Lemma
\ref{max_ineq}, the sixth moment of $\lambda^{-1/4}J(\lambda t)$ is
uniformly bounded. So, we have uniform integrability for first and
second powers of the process (see also De Masi-Ferrari \cite{DeMF}),
and therefore convergence of these moments.

To finish, we
just have to compute the limit of covariances
\begin{eqnarray*}
&&4{\rm cov}\Big(\lambda^{-1/4}J(\lambda t),\lambda^{-1/4}J(\lambda s)\Big)\\
&&\ \ \ = \ E\Big[\lambda^{-1/4}J(\lambda t)+
\lambda^{-1/4}J(\lambda
s)\Big]^{2}-E\Big[\lambda^{-1/4}J(\lambda
t)-\lambda^{-1/4}J(\lambda s)\Big]^{2}
\\
&&\ \ \ = \ E\Big[\lambda^{-1/2}J^{2}(\lambda t)\Big]
+E\Big[\lambda^{-1/2}J^{2}(\lambda t)\Big]\\
&&\ \ \ \ \ \ \ \ \ \ \  +\   2{\rm cov}\Big(\lambda
^{-1/4}J(\lambda t),\lambda^{-1/4}J(\lambda s)\big)
-E\Big[\lambda^{-1/2}J^{2}(\lambda(t-s))\Big].
\end{eqnarray*}
Then, for $t>s$, recalling (\ref{subdifflimits}),
\begin{equation*}
\lim_{\lambda\rightarrow \infty }{\rm
cov}\Big(\lambda^{-1/4}J(\lambda t),\lambda^{-1/4}J(\lambda s)\Big)\
=\ \frac{\sigma_J^{2}}{2}\Big(\sqrt{t}+\sqrt{s}-\sqrt{t-s}\Big).
\end{equation*}
\end{proof}

\section{Approximation and fBM limit for the tagged particle}
\label{tagged_sect}
In this section, we approximate $\lambda^{-1/4}X(\lambda t)$ by
$\lambda^{-1/4}\rho^{-1}J(\lambda t)$ (Proposition
\ref{tagged_approx}) by adapting part of the proof of Proposition 2.8 in
D\"urr-Goldstein-Lebowitz \cite{DGL}.  Hence, as a process,
$\lambda^{-1/4}X(\lambda t)$ will converge to the same fractional Brownian motion
limit as $\lambda^{-1/4}\rho^{-1}J(\lambda t)$ (Corollary \ref{tagged_fbm}).
\medskip

The first step is to approximate on the integers.
For $0<\epsilon<1$, $k\geq 1$ and $t\geq 0$, define
$$J_{\epsilon, k}(t) \ = \ \left\{\begin{array}{rl} \epsilon
|J(t)| & \ {\rm for \ } |J(t)| \geq t^{1/8} +k\\
3\rho^{-1}(t^{1/8} +k) & \ {\rm for \ } |J(t)|< t^{1/8}
+k.\end{array}\right.$$
Recall from Subsection \ref{tagged_subsection}, the tagged particle representation $X(t) = Y_{J(t)}$.
\begin{proposition}
\label{tagged_disc} For $0< \epsilon<1$,
we have
$$\lim_{k\rightarrow \infty} P\Big ( \sup_{t\in \bb Z_+} \frac{|Y_{J(t)}(t) -
\rho^{-1}J(t)|}{J_{\epsilon,k}(t)} >1\Big) \ = \ 0.$$
\end{proposition}

\begin{proof} First, we note, as in D\"urr-Goldstein-Lebowitz
\cite{DGL}, on the event
$$G_{\epsilon,m} \ =\  \Big\{ |Y_n(t) - \rho^{-1}n| \leq \epsilon
|n| \ {\rm for \ all \ }|n|\geq m\Big\}$$ that
$|Y_n(t) -\rho^{-1}n| \leq   3\rho^{-1}m$
when $|n|\leq m$ because particles cannot cross and so
\begin{equation}
\label{rmk_G}|Y_n(t)| \ \leq\
\max\{|Y_{-m}(t)|,|Y_m(t)|\} \ \leq \ m(\rho^{-1} +\epsilon) \leq
2\rho^{-1}m.
\end{equation}

By the representation from Subsection 2.3, and that $\{d_i(t): i\neq
0\}$, $Y_1(t)$ and $|Y_0(t)|+1$ are independent Geometric$(\rho)$
random variables, we deduce
\begin{eqnarray*}
&&P\Big( |Y_n(t) - \rho^{-1}n|> |n|\epsilon \ {\rm for \ some \
}|n|\geq m\Big)\\
 &&\ \ \ \ \leq \
2\sum_{l\geq m} P\Big( \Big|(Y_1(t)-\rho^{-1})+\sum_{i=1}^{l-1} (d_i(t)
-\rho^{-1})\Big| > l\epsilon\Big)\\
&&\ \ \ \ \leq \  C_6 m^{-q}
\end{eqnarray*}
where $C_6=C_6(\epsilon, \rho, q)$ for any power of $q>0$.

Hence, noting the remark made near (\ref{rmk_G}),
\begin{eqnarray*}
&&P\Big( {\rm for \ some \ }t\in \bb Z_+, \ |Y_{J(t)}(t) - \rho^{-1}J(t)|> |J_{\epsilon,k}(t)|\Big)\\
&&\ \ \leq 2\ \sum_{t\in \bb Z_+} P\Big( |Y_n(t) -
\rho^{-1}n|>\epsilon |n| \ {\rm for \
some \ }|n|\geq t^{1/8}+k\Big)\\
&&\ \ \leq \ C_6 \sum_{t\in \bb Z_+}(t^{1/8}+k)^{-q}.
\end{eqnarray*}
With $q>8$,
the last expression vanishes as $k\uparrow \infty$.
\end{proof}

\begin{proposition}
\label{tagged_approx}
For $\delta>0$, we have
$$\lim_{\lambda\rightarrow \infty} P\Big (
\sup_{t\in [0,1]} \lambda^{-1/4}|X(\lambda t) - \rho^{-1}J(\lambda
t)|>\delta\Big) \ =\ 0.$$
\end{proposition}

\begin{proof}
First, we write
\begin{eqnarray*}
|X(\lambda t) - \rho^{-1}J(\lambda t)| &\leq& |X(\lambda t)-X(\lfloor\lambda t\rfloor)|
+|X(\lfloor\lambda t\rfloor) - \rho^{-1}J(\lfloor\lambda t\rfloor)|\\
&&\ \ \ \ \ \ \ + \ \rho^{-1}|J(\lfloor\lambda
t\rfloor)-J(\lambda t)|\end{eqnarray*}
 and note both
$$\lim_{\lambda \rightarrow\infty}P\Big( \sup_{t\in [0,1]}
\lambda^{-1/4}|X(\lambda t)-X(\lfloor\lambda
t\rfloor)|>{\delta}/{3}\Big)
\ =\ 0$$
and
$$\lim_{\lambda \rightarrow\infty}P\Big( \sup_{t\in [0,1]}
\lambda^{-1/4}|J(\lambda t)-J(\lfloor\lambda t\rfloor)|>{\delta}/{3}\Big) \ =\
0.$$
Indeed, the second limit is estimated in the proof of
Proposition \ref{J_diff}, and the first limit is similarly argued:
Write $X(t)-X(0) = \mc M(t) + \mc A(t)$ where $\mc M(t)$ is a
martingale and $\mc A(t)$ is an additive functional with integrand
bounded by $1/2$; note $|X(0)|+1$ is a Geometric$(\rho)$ random
variable, and so $E[X^4(0)] \leq C_7$; then,
\begin{eqnarray*}
E\big[\sup_{0\leq t\leq 1}X^4(t)\big] & \leq &
4E\big[\sup_{0\leq t\leq 1} (X(t)-X(0))^4\big]
+ 4 E[X^4(0)]\\
&\leq& 16(4/3)^4\big(E\big[M^4(1)\big] + 1\Big) +4C_7 \\
& \leq & C_8 \Big( E\big[X^4(1)\big] + 1\Big);\end{eqnarray*} and,
as $X(t)-X(0) = \mc N_+(t) - \mc N_-(t)$ is the difference of two
counting processes whose infinitessimal rates are bounded by $1/2$, by
coupling with respect to dominating Poisson rate$(1/2)$ processes,
$E[X^4(1)]< 4E[(X(1)-X(0))^4] + 4C_7<\infty$.

Hence, we need only show
$$\lim_{\lambda \rightarrow\infty}P\Big( \sup_{t\in [0,1]}
\lambda^{-1/4}|X(\lfloor\lambda t\rfloor)-\rho^{-1} J(\lfloor\lambda
t\rfloor)|>{\delta}/3\Big) \ = \ 0.$$ This limit is the same as
$$ \mf{L}\ :=\ \lim_{\lambda\rightarrow
\infty} P\Big (\max_{0\leq l\leq \lfloor\lambda\rfloor}
\lambda^{-1/4}|X(l) - \rho^{-1}J(l)|>\delta/3\Big).$$
 As $X(l)=Y_{J(l)}(l)$, and
$$\lambda^{-1/4}|Y_{J(l)}(l)-\rho^{-1}J(l)|\  =\  \frac{|Y_{J(l)}(l) -
\rho^{-1}J(l)|}{|J_{\epsilon,k}(l)|}\frac{|J_{\epsilon,k}(l)|}{\lambda^{1/4}}$$
for $0<\epsilon<1$ and $k\geq 1$,
noting for large $\lambda$
that $3\rho^{-1}(\lambda^{1/8}+k)\lambda^{-1/4}<\delta/6$, we have
\begin{eqnarray*}\mf{L}&\leq&\lim_{\lambda\rightarrow \infty}
P\Bigg(\max_{0\leq l\leq \lfloor\lambda\rfloor}
\frac{|J_{\epsilon,k}(l)|}{\lambda^{1/4}}>\frac{\delta}{3}\Bigg) +
P\Bigg ( \sup_{t\in \bb Z_+} \frac{|Y_{J(t)}(t) -
\rho^{-1}J(t)|}{J_{\epsilon,k}(t)} >1\Bigg).
\end{eqnarray*}
With $\epsilon$ fixed for the moment, $\mf{L}$ is further bounded by
$$\lim_{\lambda\rightarrow\infty} P\Bigg( \sup_{t\in
[0,1]} \frac{|J(\lambda t)|}{\lambda^{1/4}}\geq
\frac{\delta}{3\epsilon}\Bigg) + \lim_{k\rightarrow \infty}
P\Bigg ( \sup_{t\in \bb Z_+} \frac{|Y_{J(t)}(t) -
\rho^{-1}J(t)|}{J_{\epsilon,k}(t)} >1\Bigg).
$$
The
second limit vanishes by Proposition \ref{tagged_disc}.  Also,
the first limit, by the
invariance principle already proved for $\lambda^{-1/4}J(\lambda t)$
with respect to continuous fractional Brownian motion
(Corollary \ref{J_fbm}), vanishes by later taking
$\epsilon\downarrow 0$. \end{proof}

As mentioned in the beginning of the section, from Proposition
\ref{tagged_approx} and Corollary \ref{J_fbm}, we obtain the
fractional Brownian motion limit for the rescaled tagged motion.
\begin{corollary}
\label{tagged_fbm}
Under initial distribution $\nu_\rho$,
\begin{equation*}
\lambda^{-1/4}X(\lambda t) \ \Rightarrow \ \sigma_X \bb B_{1/4}(t)
\end{equation*}
in $D([0,1])$ endowed with the uniform topology, where $\bb
B_{1/4}(t)$ is the fractional Brownian motion process with parameter
$1/4$.
\end{corollary}



\end{document}